\documentclass{procpost}
\usepackage{amssymb}

\newcommand{\emp}{\emptyset}
\newcommand{\ben}{\mathbb N}
\newcommand{\ber}{\mathbb R}
\newcommand{\beq}{\mathbb Q}
\newcommand{\bez}{\mathbb Z}

\newcommand{\nhat}[1]{\{1,2,\ldots,#1\}}
\newcommand{\ohat}[1]{\{0,1,\ldots,#1\}}

\newcommand{\pf}{{\mathcal P}_f}
\newcommand{\symdif}{\bigtriangleup}

\DeclareMathOperator{\supp}{supp}

\newtheorem{theorem}{Theorem}[section]
\newtheorem{lemma}[theorem]{Lemma}
\newtheorem{corollary}[theorem]{Corollary}
\newtheorem{conjecture}[theorem]{Conjecture}

\theoremstyle{definition}
\newtheorem{definition}[theorem]{Definition}

\theoremstyle{remark}
\newtheorem{remark}[theorem]{Remark}

\numberwithin{equation}{section}

\begin{document}

\title[Partition regularity]{Partition regularity without the columns property}

\author{Ben Barber}
\address{School of Mathematics,
                  University of Birmingham,
                  Edgbaston,
                  Birmingham, B15 2TT, UK.}
\email{b.a.barber@bham.ac.uk}

\author{Neil Hindman}
\address{Department of Mathematics,
                  Howard University,
                  Washington, DC 20059, USA.}
\email{nhindman@aol.com}
\thanks{The second author acknowledges support received from the National
                Science Foundation (USA) via Grant DMS-1160566.}

\author{Imre Leader}
\address{Department of Pure Mathematics and Mathematical Statistics,
                  Centre for Mathematical Sciences,
                  Wilberforce Road, Cambridge, CB3 0WB, UK.}
\email{i.leader@dpmms.cam.ac.uk}

\author{Dona Strauss}
\address{Department of Pure Mathematics,
                  University of Leeds, Leeds LS2 9J2, UK.}
\email{d.strauss@hull.ac.uk}

\begin{abstract}

A finite or infinite matrix $A$ with rational entries is called {\it partition
regular\/} if whenever the natural numbers are finitely coloured there is a 
monochromatic vector $x$ with $Ax=0$. Many of the classical theorems of Ramsey
Theory may naturally be interpreted as assertions that particular matrices
are partition regular. In the finite case, Rado proved that a matrix is
partition regular if and only it satisfies a computable condition known as the 
{\it columns property\/}. The first requirement
of the columns property is that some set of columns sums to zero.

In the infinite case, much less is known. There are many examples of matrices 
with the columns property that are not partition regular, but until now all
known examples of partition regular matrices did have the columns property. Our 
main
aim in this paper is to show that, perhaps surprisingly, there are infinite
partition regular matrices without the columns property --- in fact, having no
set of columns summing to zero.

We also make a conjecture that if a partition regular matrix (say with integer
coefficients) has bounded row
sums then it must have the columns property, and prove a first step towards
this.
\end{abstract}

\maketitle

\section{Introduction}

Let $A$ be an $u \times v$ matrix with rational entries. We say that $A$ is
{\it kernel partition regular}, or simply {\it partition regular}, if for 
every finite colouring of the natural
numbers $\ben =\{ 1,2,\ldots \}$ there is a monochromatic vector $\vec x \in 
\ben^v$
with $A\vec x=\vec 0$. In other words, $A$ is partition regular if for every positive
integer $k$, and every function $c:\ben \rightarrow \{1,\ldots,k \}$, there is
a vector $\vec x=(x_1,\ldots,x_v) \in \ben^v$ with
$c(x_1)=\ldots=c(x_v)$ such that $A\vec x=\vec 0$. We may also speak of the `system of
equations $A\vec x=\vec 0$' being partition regular.

Many of the classical results of Ramsey Theory may naturally be considered
as statements about partition regularity.
For example, Schur's Theorem \cite{S}, that in any finite
colouring of the natural numbers we may solve $x+y=z$ in one colour class, is
precisely the assertion that the $1 \times 3$ matrix $(1, 1, -1)$ is partition
regular. As another example, the theorem of van der Waerden \cite{W} 
that, for any
$m$, every finite colouring of the natural numbers contains a monochromatic
arithmetic progression with $m$ terms, is (with the strengthening that we may
also choose the common difference of the sequence to have the same colour)
exactly the statement that the $(m-1) \times (m+1)$ matrix 
\[
\begin{pmatrix}
1&1&-1&0&\ldots&0&0\\
1&0&1&-1&\ldots&0&0\\
1&0&0&1&\ldots&0&0\\
\vdots&\vdots&\vdots&\vdots&\ddots&\vdots&\vdots\\
1&0&0&0&\ldots&1&-1
\end{pmatrix}\]
is 
partition regular.

In 1933 Rado characterized those matrices that are partition regular
over $\ben$ in terms of the {\it columns property\/}.

\begin{definition}\label{defcc} Let $u,v\in\ben$ and let $A$ be a $u\times v$ matrix with entries 
from $\beq$.  Denote the columns of $A$ by $\langle\vec c_i\rangle_{i=0}^{v-1}$.
  The matrix $A$ satisfies the {\it columns property\/} if and only if
there exist $m\in\nhat{v}$ and a partition $\langle I_t\rangle_{t=0}^{m-1}$ of 
$\ohat{v-1}$ such that
\begin{enumerate}
  \item[(1)] $\sum_{i\in I_0}\vec c_i \ = \ \vec 0$ and
  \item[(2)] for each $t\in \nhat{m-1}$, $\sum_{i\in I_t}\vec c_i$ is a linear combination with coefficients from $\beq$ of $\{\vec c_i:i\in\bigcup_{j=0}^{t-1}I_j\}$.
\end{enumerate}

\end{definition}

\begin{theorem}\label{Rado} Let $u,v\in\ben$ and let $A$ be a $u\times v$ matrix with entries 
from $\beq$. Then $A$ is partition regular if and only if $A$ satisfies the 
columns property.\end{theorem}

\begin{proof} \cite[Satz IV]{R}.\end{proof}

In this paper, we are concerned with partition regularity of infinite systems of 
homogeneous linear equations or, equivalently, with partition regularity of
infinite matrices.  In all cases we shall assume that the equations we are dealing
with have finitely many terms; that is, each row of the matrix of coefficients
has all but finitely many entries equal to $0$.  The definition of partition regularity for a matrix (or a system of linear equations)
is verbatim the same
as for finite matrices or systems of equations. 

The first example of a (non-trivial) infinite
partition regular system of
equations was constructed in 1974 \cite{Ha}, proving a conjecture
of Graham and Rothschild: in any finite colouring of the natural numbers there is a 
sequence $x_1,x_2,\ldots$ of natural numbers such that the set
$$FS(x_1,x_2,\ldots)= \{ \sum_{i \in I}x_i:\; I \subseteq \ben,\,  I 
\hbox{ finite and non-empty } \}$$
is monochromatic. This is also known as the Finite Sums Theorem. (It is worth
remarking that the finite analogue of this, known as Folkman's Theorem,
stating that, for any $m$, in any finite colouring of the
natural numbers we may find $x_1,\ldots,x_m$ with $FS(x_1,\ldots,
x_m)$ monochromatic, follows easily from Rado's Theorem.) Since then, several
other infinite partition systems have been found --- for example, the 
Milliken--Taylor Theorem \cite{M} \cite{T} and several systems in \cite{HLSaa}.
See \cite{HLS} for general background on this. 

We take $\omega=\ben \cup \{0\}$.  The columns property has an obvious 
extension to 
infinite matrices.

\begin{definition}\label{defccinf} Let $u,v\in\ben\cup\{\omega\}$ 
and let $A$ be a $u\times v$ matrix with entries 
from $\beq$.  Denote the columns of $A$ by $\langle\vec c_i\rangle_{i<v}$.
  The matrix $A$ satisfies the {\it columns property\/} if and only if
there exists a partition $\langle I_\sigma\rangle_{\sigma<\mu}$ of 
$v$, where $\mu \in\ben\cup\{\omega\}$, such that
\begin{itemize}
\item[(1)] $\sum_{i\in I_0}\vec c_i \ = \ \vec 0$ and
\item[(2)] for each $t\in \mu\setminus\{0\}$, 
$\sum_{i\in I_t}\vec c_i$ is a linear combination with coefficients
from $\beq$ of $\{\vec c_i:i\in\bigcup_{j<t}I_j\}$.
\end{itemize}
\end{definition}

We stress that the sums in this definition are {\it not} required to be
finite. Note that infinite sums of columns do always make sense, because our
matrices have only finitely many non-zero entries in each row. It is also worth
remarking that if one insisted on finite sums then nothing would work: even the
matrix corresponding to the Finite Sums Theorem needs infinite sums for the
columns property.

It is easy to see that, for infinite matrices, the columns property is not
sufficient for partition regularity.  Consider for example the system
of equations $x_n-x_{n+1}=y_n, \ n \in \omega$.  As a matrix equation this corresponds to
\[
\begin{pmatrix}
1&-1&-1&0&0&0&0&\cdots\\
0&0&1&-1&-1&0&0&\cdots\\
0&0&0&0&1&-1&-1&\cdots\\
\vdots&\vdots&\vdots&\vdots&\vdots&\vdots&\vdots&\ddots
\end{pmatrix}
\begin{pmatrix}
x_0\\ y_0\\ x_1\\ y_1\\ \vdots\end{pmatrix}=\vec 0.
\]
The system of equations is not partition regular over $\ben$ for the
trivial reason that it has no solutions there at all; any solution must have
$x_n>x_{n+1}$ for each $n$.
To see that the matrix satisfies the columns property, let
$I_0=\{0,2,4,6,\ldots\}$ and $I_1=\{1,3,5,\ldots\}$.
Then $\sum_{i\in I_0}\vec c_i=\vec 0$ and $\sum_{i\in I_1}\vec c_i=\sum_{t=1}^\infty t\vec c_{2t}$.

We shall show in Section~\ref{prnotcc} of this paper there is a system of 
linear equations which is 
partition regular over $\ben$ but for which the coefficient matrix has no set of columns summing to zero. We mention in passing that the entries and the columns
of this matrix will be `nicely behaved': all entries are from the set
$\{ -1,0,1,2 \}$ and each column of the coefficient matrix has at most 
three nonzero entries. 

The example mentioned in the previous paragraph has the property that the row sums
of the coefficient matrix are unbounded.  We also show in Section~\ref{prnotcc} that if 
$A$ is a matrix which is partition regular over $\ben$ and the
sums of the absolute values of the entries in any row are bounded, then there
must be some set of columns that sum to zero.  We do not know whether such a 
matrix must satisfy the full columns property.

The proof that the system mentioned above is partition regular over $\ben$ depends 
on the following lemma, which is based on results in \cite{BHL}. The lemma refers
to {\it central\/} sets.  From the point of view of the partition regularity of 
our system of equations, central sets have two key properties.
\begin{itemize}
\item Whenever $\ben$ is finitely
coloured, one of the colour classes must be central.
\item Any finite partition regular system of homogeneous linear equations has solutions
in any central set (see \cite[Theorem 15.16(b)]{HS}).  (The corresponding statement is not true for infinite systems.)
\end{itemize}

In the statement of this lemma, $m\cdot\bez=\{m\cdot n:n\in\bez\}$ while
$kC=C+C+\cdots+C$ ($k$ times).

\begin{lemma}\label{AminuskA} Let $C$ be a central subset of $\ben$.  There exist
$m$ and $K$ in $\ben$ such that if $k\geq K$, then $m\cdot\bez\subseteq C-kC$.
\end{lemma}

Section~\ref{proof-of-AminuskA} consists of a proof of (a generalisation of) Lemma \ref{AminuskA}, and also  
a strengthening of the main result of Section~\ref{prnotcc} (namely
Theorem \ref{iskprn}).
The proofs depend on the algebraic structure of the Stone--\v Cech compactification
$\beta\ben$ of the discrete space $\ben$.  We give a brief introduction to
that structure in that section. We mention that the point of Lemma \ref{AminuskA}, and the point of using 
central sets, is that then the proof of Theorem \ref{iskprn} is clean and direct.

\section{Partition regularity does not 
imply the\\ columns property}
\label{prnotcc}

In this section we consider the following system of equations:
\[
  2x_n+x_{2^n}+x_{2^n+1}+\cdots+x_{2^{n+1}-1}=y_n, \ \ \ n \in \omega
\]
The simplest way to represent the matrix of coefficients is as an $\omega\times(\omega+\omega)$
matrix.  That is, we let $A$ be the $\omega\times\omega$ matrix such that, for
$i,j\in\omega$,
\[
a_{i,j} =
  \begin{cases}
    2 & \text{if } j = i, \\
    1 & \text{if } 2^i \leq j < 2^{i+1}, \\
    0 & \text{otherwise,}
  \end{cases}
\]
so that 
\[
A=\begin{pmatrix}
2&1&0&0&0&0&0&0&0&\cdots\\
0&2&1&1&0&0&0&0&0&\cdots\\
0&0&2&0&1&1&1&1&0&\cdots\\
\vdots&\vdots&\vdots&\vdots&\vdots&\vdots&\vdots&\vdots&\vdots&\ddots
\end{pmatrix}.
\]
Letting $I$ be the $\omega \times \omega$ identity matrix, we
then have that the matrix of coefficients for our system of equations
is $\begin{pmatrix}A&-I\end{pmatrix}$.  It is 
immediate that there is no non-empty set of columns of this matrix summing to $\vec 0$.

\begin{theorem}\label{iskprn} For any central subset $C$ of $\ben$,
there exist $\vec x$ and $\vec y$ in $C^\omega$ such that
$\begin{pmatrix}A&-I\end{pmatrix}\begin{pmatrix} \vec x\\ 
\vec y\end{pmatrix} = \vec 0$.  In particular the system of equations
\[
  2x_n + x_{2^n} + x_{2^n+1} + \cdots + x_{2^{n+1}-1} = y_n, \ \ \ n \in \omega
\]
is partition regular over $\ben$.
\end{theorem}

\begin{proof}   Pick $m$ and $K$ as guaranteed for $C$ by Lemma \ref{AminuskA}.
Pick $M\in\ben$ such that $2^{M+1}\geq m-2+K$.  Let
$P$ be the $(M+1)\times 2^{M+1}$ matrix consisting of rows 
$0,1,\ldots,M$ and columns $0,1,\ldots,2^{M+1}-1$ of $A$ and
let $I_{M+1}$ be the $(M+1)\times(M+1)$ identity matrix.  Then
$\begin{pmatrix}P&-I_{M+1}\end{pmatrix}$ satisfies the
columns property.  Since, as we remarked earlier, any 
finite partition regular system of homogeneous linear equations has solutions
in any central set,  we may choose $\vec x=\langle x_i\rangle_{i=0}^{2^{M+1}-1}$
and $\vec y=\langle y_i\rangle_{i=0}^M$ in $C$ such that
$\begin{pmatrix}P&-I_{M+1}\end{pmatrix}\begin{pmatrix}\vec x\\
\vec y\end{pmatrix}=\vec 0$.
That is, for each $n\in\ohat{M}$, the equation 
$2x_n+x_{2^n}+x_{2^n+1}+\cdots+x_{2^{n+1}-1}=y_n$ holds.

Let $r\geq M$ and assume that we have chosen $\langle x_i\rangle_{i=0}^{2^{r+1}-1}$
and $\langle y_i\rangle_{i=0}^r$ in $C$ such that for each $n\in\ohat{r}$, the equation 
$2x_n+x_{2^n}+x_{2^n+1}+\cdots+x_{2^{n+1}-1}=y_n$ holds.
For $2^{r+1} \leq t \leq 2^{r+1}+m-3$, let $x_t=x_{r+1}$. Let $k=2^{r+1}-m+2$.
Then $k\geq K$ and 
\[
2x_{r+1}+x_{2^{r+1}}+x_{2^{r+1}+1}+\cdots+x_{2^{r+1}+m-3}=
m\cdot x_{r+1}\in m\cdot \bez
\]
so by Lemma \ref{AminuskA} we may pick  $y_{r+1}\in C$
and pick $x_t\in C$ for 
$2^{r+1}+m-2\leq t \leq 2^{r+2}-1$  such that
$2x_{r+1}+x_{2^{r+1}}+x_{2^{r+1}+1}+\cdots+x_{2^{r+2}-1}=y_{r+1}$.
\end{proof}

We shall see in Corollary \ref{strongcentral} that in fact one may choose
$\vec x$ and $\vec y$ in $C^\omega$ such that
$\begin{pmatrix}A&-I\end{pmatrix}\begin{pmatrix} \vec x\\ 
\vec y\end{pmatrix}=\vec 0$ and 
all entries of $\begin{pmatrix} \vec x \\ \vec y \end{pmatrix}$ are distinct.

The matrix of Theorem \ref{iskprn} has unbounded row sums.This motivates the
following. Let us say
that an $\omega\times\omega$ matrix $A$ with
entries from $\bez$ has {\it bounded row sums} if
$\{\sum_{j<\omega}|a_{i,j}|:i<\omega\}$ is bounded. 
 
We now see
that if $A$ is a matrix which which is partition regular and has bounded row 
sums then in fact some non-empty set of
columns must sum to zero. In fact, this is pretty much a direct copy of Rado's
original proof in the finite case.

\begin{theorem}\label{bounded} Let $A$ be an $\omega\times\omega$ matrix with
entries from $\bez$ having bounded row sums. If
$A$ is $KPR/\ben$ and $\vec c_i$ denotes column $i$ of $A$, then there
is some $J\subseteq\omega$ such that $\sum_{i\in J}\vec c_i=\vec 0$.
\end{theorem}

\begin{proof} Pick a prime $q$ such that $\sum_{j<\omega}|a_{i,j}|<q$ for
each $i<\omega$.  Given $x\in\ben$, pick $b(x)\in\nhat{q-1}$, $l(x)\in\omega$ 
and $a(x)\in\omega$ such that $x=b(x)q^{l(x)}+a(x)q^{l(x)+1}$.  (Thus
$b(x)$ is the rightmost nonzero digit in the base $q$ expansion of $x$ 
and $l(x)$ is the position of that digit.)  For $b\in\nhat{q-1}$, let
$C_b=\{x\in\ben:b(x)=b\}$.  Pick $b\in\nhat{q-1}$ and 
$\vec x\in C_b^\omega$ such that $A\vec x=\vec 0$.  Let $d=\min\{l(x_i):i<\omega\}$
and let $J=\{i<\omega:l(x_i)=d\}$.  We shall show that $\sum_{j\in I}\vec c_j=\vec 0$.

For $j\in J$, let $e_j=a(x_j)$ so that $x_j=bq^d+e_jq^{d+1}$.  For $j\in\omega\setminus J$,
$l(x_j)>d$ so pick $e_j\in\ben$ such that $x_j=e_jq^{d+1}$.
Suppose that $\sum_{j\in J}\vec c_j\neq \vec 0$ and pick some $i<\omega$
such that $\sum_{j\in J}a_{i,j}\neq 0$. Then
$0=\sum_{j\in J}bq^{d}a_{i,j}+\sum_{j<\omega}e_jq^{d+1}a_{i,j}$ so
$q$ divides $\sum_{j\in J}ba_{i,j}$ so 
$q$ divides $\sum_{j\in J}a_{i,j}$ while $q>\sum_{j\in J}|a_{i,j}|$, a contradiction.
\end{proof} 

\begin{remark} A matrix $A$ which satisfies the weaker condition that 
$\{|\sum_{j<\omega}a_{i,j}|:i<\omega\}$ is bounded 
could be partition regular and have no non-empty set of columns summing to 
zero. If
\[
A=\begin{pmatrix}
-2&1&0&0&0&0&0&0&\cdots\\
0&-3&1&1&0&0&0&0&\cdots\\
0&0&-4&0&1&1&1&0&\cdots\\
\vdots&\vdots&\vdots&\vdots&\vdots&\vdots&\vdots&\vdots&\ddots
\end{pmatrix}
\]
then  $\begin{pmatrix}A&-I\end{pmatrix}$ is partition regular; the proof is essentially the same as that of
Theorem \ref{iskprn}. Clearly, there is no non-empty set of columns of $\begin{pmatrix}A & -I\end{pmatrix}$ whose sum is zero.
\end{remark}

Very frustratingly, we have been unable to extend Theorem \ref{bounded} to 
the following.

\begin{conjecture}\label{colsbounded} Let $A$ be an $\omega\times\omega$ matrix with
entries from $\bez$, and assume
that $A$ is partition regular. If $A$ has bounded row sums, then $A$ has the columns property.
\end{conjecture}

If this conjecture is false, might at least a weaker statement be true, that
$A$ must satisfy the natural `transfinite' version of the columns property?

\begin{definition}\label{defccinfb} Let $u,v\in\ben\cup\{\omega\}$ 
and let $A$ be a $u\times v$ matrix with entries 
from $\beq$.  Denote the columns of $A$ by $\langle\vec c_i\rangle_{i<v}$.
  The matrix $A$ satisfies the {\it transfinite columns property\/} if and only if
there exist a countable ordinal $\mu$ and a partition $\langle I_\sigma\rangle_{\sigma<\mu}$ of 
$v$ such that
\begin{itemize}
\item[(1)] $\sum_{i\in I_0}\vec c_i=\vec 0$ and
\item[(2)] for each $t\in \mu\setminus\{0\}$,  
$\sum_{i\in I_t}\vec c_i$ is a linear combination with coefficients
from $\beq$ of $\{\vec c_i:i\in\bigcup_{j<t}I_j\}$.
\end{itemize}
\end{definition}

We remark that, for any given countable ordinal $\mu$, it is a straightforward exercise to write down
a matrix that has the transfinite columns property with the partition being
indexed by $\mu$, but without any partition having a smaller indexing ordinal.

Our weaker version of Conjecture \ref{colsbounded} would then be:

\begin{conjecture}\label{weakercolsbounded} Let $A$ be an 
$\omega\times\omega$ matrix with
entries from $\bez$ and assume
that $A$ is partition regular. If $A$ has bounded row sums, then $A$ has the transfinite columns property.
\end{conjecture}

We saw in the Introduction that the system of equations $x_n-x_{n+1}=y_n$
satisfies the columns property but is not partition regular over $\ben$.  
However, as shown at the end of Section~\ref{proof-of-AminuskA} of \cite{HLS}, it is partition
regular over $\ber$ --- meaning that whenever $\ber \setminus \{ 0 \}$ is
finitely coloured there is a monochromatic solution to that system of
equations. (It was also shown there that it is not partition 
regular over $\beq$.)  This leaves open the possibility that perhaps a sufficiently
well behaved system of equations which satisfies the columns property
must at least be partition regular over $\ber$. To end this section, we show
that this seems not to be the case. The system $x_n+2x_{n+1}=y_n$ is about
as well behaved as one can wish.  The sum of absolute values of entries
of each row of the coefficient matrix is $4$ and all columns sum to
$1$, $-1$, or $3$. The matrix equation is
\[
\begin{pmatrix}
1&-1&2&0&0&0&0&0&\cdots\\
0&0&1&-1&2&0&0&0&\cdots\\
0&0&0&0&1&-1&2&0&\cdots\\
\vdots&\vdots&\vdots&\vdots&\vdots&\vdots&\vdots&\vdots&\ddots\end{pmatrix}
\begin{pmatrix}x_0\\y_0\\x_1\\y_1\\
\vdots\end{pmatrix}=\vec 0.
\]
The matrix of coefficients satisfies the columns property with $I_i=\{2i,2i+1\}$ for
each $i<\omega$.

\begin{theorem}\label{notkprr} The system of equations $x_n+2x_{n+1}=y_n$ is not
partition regular over $\ber$.\end{theorem}

\begin{proof} We show first that the system is not partition regular over $\beq$, for
which it suffices to show that it is not partition regular over $\beq^+=\{x\in\beq:x>0\}$.
We begin by defining $\tau:\beq^+\to\{0,1,2\}$ by $\tau(x)=i$ if and only
if $\lfloor \log_2(x)\rfloor\equiv i\ \pmod 3$.  Note that if
$\tau(x)=\tau(y)$ and $y\geq 2x$, then $y>4x$.  (To see this, let
$k=\lfloor \log_2(x)\rfloor$ and $m=\lfloor \log_2(y)\rfloor$. Then
$2^{m+1}>y\geq 2x\geq 2^{k+1}$ so $m>k$ and therefore 
$m\geq k+3$.  Then $y\geq 2^m\geq 2^{k+3}>4x$.)

As is relatively well known, and at any rate easy to verify,
every rational $x\in(0,1)$ has a unique expansion of the form
\[
x = \sum_{t=2}^{m(x)} \frac{a(x,t)}{t!}
\]
where each $a(x,t)\in\ohat{t-1}$ and $a\big(x,m(x)\big)>0$. 
If $t>m(x)$, let $a(x,t)=0$. For
each $t\geq 2$, choose $\nu_t:\nhat{t-1}\to\{0,1,2\}$ such that,
if $i,j\in\nhat{t-1}$ and $j\equiv 2i\ \pmod t$, then
$\nu_t(i)\neq\nu_t(j)$.  (To see that such a colouring exists, let $G$ be the digraph on vertex set $\nhat{t-1}$ with an edge from $i$ to $j$ whenever $j \equiv 2i \pmod t$.  Then every edge of $G$ is in at most one cycle, so $G$ can be 3-coloured.)

Now define a finite colouring $\varphi$ of $\beq^+$ so that
for $x,y\in \beq^+$, $\varphi(x)=\varphi(y)$ if and only if
\begin{itemize}
\item[(1)] $\tau(x)=\tau(y)$;
\item[(2)] $x<1$ if and only if $y<1$; and
\item[(3)] if $x<1$ and $y<1$ then
  \begin{itemize}
    \item[(a)] $m(x)\equiv m(y)\ \pmod 3$ and
    \item[(b)] if $m(x)=m(y)$, then $\nu_{m(x)}\big(a(x,m(x)\big)\big)=
               \nu_{m(y)}\big(a(y,m(y)\big)\big)$.
  \end{itemize}
\end{itemize}

Now suppose we have a system $x_n+2x_{n+1}=y_n$ for $n<\omega$ such
that $\{x_n:n<\omega\}\cup\{y_n:n<\omega\}$ is monochromatic with
respect to $\varphi$.  Given $n$, we have that $y_n>2x_{n+1}$ so,
since $\tau(y_n)=\tau(x_{n+1})$, we have $y_n>4x_{n+1}$.  But then
$x_n=y_n-2x_{n+1}>2x_{n+1}$.  Therefore the sequence
$\langle x_n\rangle_{n=0}^\infty$ is eventually less than $1$ and
therefore $\{x_n:n<\omega\}\cup\{y_n:n<\omega\}\subseteq(0,1)$.
Given any $k$, $\{x\in(0,1):m(x)\leq k\}$ is finite and therefore,
there is some $n$ such that $m(x_{n+1})>m(x_n)$ and consequently
$m(x_{n+1})\geq m(x_n)+3$.  
Let $m=m(x_{n+1})$, let $(d,c,b)=\big(a(x_{n+1},m-2),a(x_{n+1},m-1),a(x_{n+1},m)\big)$,
and let $(d',c',b')=(a(y_n,m-2),a(y_n,m-1),a(y_n,m)\big)$. Now
$b'\equiv 2b\ \pmod m$ so, since $\varphi(y_n)=\varphi(x_{n+1})$ we
must have that $b'=0$ and therefore $m(y_n)<m(x_{n+1})=m$.
Thus $m(y_n)\leq m-3$ and so $d'=c'=0$.  Since $0\equiv 2b\ \pmod m$ we
must have that $m=2b$.  Since there is a carry out of position $m$
when $2x_{n+1}$ is computed,  $0=c'\equiv 2c+1\ \pmod {m-1}$ and so
there is a carry out of position $m-1$.  But then
$0=d'\equiv 2d+1\ \pmod {m-2}$, which is impossible since $m-2$ is even.

Thus we have established that our system is not partition regular over
$\beq^+$. Extend the colouring $\varphi$ to a colouring $\varphi'$ of 
$\beq\setminus\{0\}$ with respect to which there is no
monochromatic system (e.g. by reflecting to $\{x \in \beq : x < 0\}$ using a new set of colours).

 Now pick a Hamel basis $B$ for $\ber$ over $\beq$ and
well-order $B$.  For $x\in\ber\setminus\{0\}$ let $\supp(x)$ be the finite non-empty
subset of $B$ such that 
$x=\sum_{b\in\supp(x)}\alpha(x,b)b$ and each $\alpha(x,b)\in\beq\setminus\{0\}$
and let $b(x)=\max\supp(x)$.

Now let $\psi:\beq\setminus\{0\}\to\{0,1\}$ such that for all $x\in\beq\setminus\{0\}$,
$\psi(2x)\neq\psi(x)$. (For example, colour by the parity of the number of factors of $2$ in 
the numerator or denominator of $x$.)  Now define a finite colouring $\gamma$ of
$\ber\setminus\{0\}$ so that $\gamma(x)=\gamma(y)$ if and only if
\begin{itemize}
\item[(1)] $\alpha\big(x,b(x)\big)>0$ if and only if $\alpha\big(y,b(y)\big)>0$;
\item[(2)] $\psi\big(\alpha\big(x,b(x)\big)\big)=\psi\big(\alpha\big(y,b(y)\big)\big)$; and
\item[(3)] $\varphi'\big(\alpha\big(x,b(x)\big)\big)=\varphi'\big(\alpha\big(y,b(y)\big)\big)$.
\end{itemize} 

Now suppose we have a system $x_n+2x_{n+1}=y_n$ in $\ber$ for $n<\omega$ such
that $\{x_n:n<\omega\}\cup\{y_n:n<\omega\}$ is monochromatic with
respect to $\gamma$.  If for any $n$ we have $b(x_{n+1})>b(x_n)$, then
we get $b(y_n)=b(x_{n+1})$ and $\alpha\big(y_n,b(y_n)\big)=2\alpha\big(x_{n+1},b(x_{n+1})\big)$,
contradicting requirement (2).  Therefore, for each $n$, 
$b(x_{n+1})\leq b(x_n)$.  Since there are no infinite strictly decreasing sequences in $B$, there
are some $b\in B$ and $k<\omega$ such that for all $n\geq k$, $b(x_n)=b$.  
Note that by requirement (1), for all such $n$, $b(y_n)=b$.
Now let for each $n<\omega$, $x_n'=\alpha(x_{k+n},b)$ and $y_n'=\alpha(y_{k+n},b)$.
Then $\{x_n':n<\omega\}\cup\{y_n':n<\omega\}$ is monochromatic with respect to $\varphi'$ and
for each $n$, $x_n+2x_{n+1}=y_n$, a contradiction.
\end{proof}

\section{Proof of Lemma \ref{AminuskA}} \label{proof-of-AminuskA}

In this section we will provide a proof of a generalisation of
Lemma \ref{AminuskA}. We also prove a strengthening of Theorem \ref{iskprn}.
These
results utilise the algebraic structure of the Stone--\v Cech
compactification $\beta S$ of a discrete semigroup $S$.  We shall provide
a brief introduction to this structure now. For proofs of the
assertions made here, see the first four chapters of \cite{HS}.
We shall be concerned here almost exclusively with commutative semigroups,
so we will denote the operation of $S$ by $+$. We mention here that the reader
who is not interested in the results for general groups or semigroups can skip 
Lemma \ref{dense} and its
proof and rejoin the text at Lemma \ref{syndetic}, taking $G$ to be $\bez$ and 
$d$ to be the usual density for subsets of $\bez$.

We take the points of $\beta S$ to be the ultrafilters on $S$,
the principal ultrafilters being identified with the points of $S$,
whereby we pretend that $S\subseteq\beta S$.
Given a subset $A$ of $S$, $\overline A=\{p\in\beta S:A\in p\}$ is
the closure of $A$ in $\beta S$ and $\{\overline A:A\subseteq S\}$ is 
a basis for the open sets of $\beta S$ as well as a basis for the closed
sets of $\beta S$.

The operation on $S$ extends to an operation on $\beta S$, also
denoted by $+$, with the property that for each $p\in\beta S$,
the function $\rho_p:\beta S\to\beta S$ is continuous and for
each $x\in S$, the function $\lambda_x:\beta S\to\beta S$ is continuous 
where, for $q\in \beta S$, $\rho_p(q)=q+p$ and $\lambda_x(q)=x+q$.
(The reader should be cautioned that $(\beta S,+)$ is almost 
certain to not be commutative; the centre of $(\beta\ben,+)$ is
$\ben$.) Given $A\subseteq S$ and $p,q\in\beta S$, $A\in p+q$
if and only if $\{x\in S:-x+A\in q\}\in p$, where $-x+A=\{y\in S:x+y\in A\}$.

Thus $(\beta S,+)$ is a compact Hausdorff right topological 
semigroup.  As with any such object, there exist idempotents
in $\beta S$.  Also, $\beta S$ has a smallest two sided ideal,
$K(\beta S)$, which is the union of all of the minimal
right ideals of $\beta S$ as well as the union of all of the minimal
left ideals of $\beta S$.  The intersection of any minimal right
ideal with any minimal left ideal is a group and any two such groups
are isomorphic. In particular, there are idempotents in $K(\beta S)$.
A subset $C$ of $S$ is said to be {\it central\/} if and only if
it is a member of some idempotent in $K(\beta S)$. (Such 
idempotents are said to be {\it minimal\/}.) Since whenever
the union of finitely many sets is a member of an ultrafilter, one of the sets must itself a member of the ultrafilter,
 one has immediately that whenever $S$ is 
finitely coloured, at least one colour class must be central.
A subset $A$ of $S$ is said to be {\it central*\/} if and only if
it has non-empty intersection with any central set, equivalently if
and only if it is a member of every minimal idempotent.

Given a set $X$, let $\pf(X)$ be the set of finite non-empty subsets
of $X$. A subset $A$ of $S$ is {\it syndetic\/} if and only if there
is some $L\in\pf(S)$ such that $S=\bigcup_{t\in L}(-t+A)$.
A subset $A$ of $S$ is {\it piecewise syndetic\/} if and only if
\[
\textstyle\big(\exists G\in\pf(S)\big)\big(\forall F\in\pf(S)\big)\big(\exists x\in S\big)
\big(F+x\subseteq \bigcup_{t\in G}(-t+A)\big).
\]
It is then a fact that $A$ is piecewise syndetic if and only if 
$\overline A\cap K(\beta S)\neq\emp$.  In particular,
every central set is piecewise syndetic.

\begin{lemma}\label{dense} Let $(S,+)$ be a commutative cancellative semigroup.
There is a function
$d:{\mathcal P}(S)\to[0,1]$ with the following properties.
\begin{itemize}
\item[(1)] $d(S)=1$.
\item[(2)] For each piecewise syndetic subset $A$ of $S$, $d(A)>0$.
\item[(3)] For $B,C\subseteq S$, $d(B\cup C)\leq d(B)+d(C)$.
\item[(4)] For $a\in S$ and $B\subseteq S$, $d(B)=d(a+B)=d(-a+B)$.
\item[(5)] If $B\subseteq S$, $n\in\ben$, and $x_1, x_2,\ldots,x_n\in S$ 
such that $(x_i+B)\cap (x_j+B)=\emp$ if $i\neq j$, then $d\big(\bigcup_{i=1}^n(x_i+B)\big)=nd(B)$.
\end{itemize}\end{lemma}

\begin{proof}  There are at least two ways to see this. 
One way is to use 
the following very general result, whose proof can be found in
\cite{DLS}. 

\begin{lemma}\label{dense2} Let $S$ be a left amenable semigroup and let $A$ be a piecewise syndetic
subset of $S$. Then a probability measure $\nu$ can be defined on $\beta S$ with the following properties:
\begin{itemize}
\item[(a)] $\nu(s^{-1}B)=\nu(B)$ for every Borel subset $B$ of\hfill\break
 $\beta S$ and every $s\in S$ and
\item[(b)] $\nu(\overline A)>0$.\end{itemize}
If, in addition $S$ is cancellative, then:
\begin{itemize}
\item[(c)] $\nu(sB)=\nu(B)$ for every Borel subset $B$ of $\beta S$\hfill\break
 and every $s\in S$.
\end{itemize}\end{lemma}

Using this lemma, one then defines $d(B)=\sup\{\nu(\overline B):\nu$ is a probability measure
on $\beta S$ satisfying (a) and (c) of Lemma \ref{dense2}$\}$.

An alternative, mostly elementary, method is to use the
fact proved by Argabright and Wilde
in \cite[Theorem 4]{AW} that any commutative semigroup satisfies the
{\it Strong F\o lner Condition\/}:
\[
\big(\forall H\in\pf(S)\big)\big(\forall \epsilon>0\big)
\big(\exists K\in\pf(S)\big)\big(\forall s\in H\big)\big(|K\symdif (s+K)|<\epsilon\cdot|K|\big)
\]
(For an entirely elementary proof of this fact see \cite[Section 7]{HSc}.)

If $S$ satisfies the Strong F\o lner Condition and $A\subseteq S$, then the 
{\it F\o lner density\/} of $A$ is 
 \begin{align*}
  d(A)=\sup\big\{\alpha: &
   \big(\forall H\in\pf(S)\big)\big(\forall \epsilon>0\big)
     \big(\exists K\in\pf(S)\big)\big(|A\cap K|\geq\alpha\cdot|K|\big)\text{ and }\\
 & \big(\forall s\in H\big)\big(|K\symdif (s+K)|<\epsilon\cdot|K|\big)\big)\big\}.
 \end{align*}

Conclusions (1) and (3) are routine to verify for F\o lner density and
conclusion (4) is established in \cite[Theorem 4.17]{HSb}.  Define
\[
 {\bf D}=\{p\in\beta S:(\forall A\in p)(d(A)>0)\}.
\]
To verify (2), let $A$ be a piecewise syndetic subset of $S$.
By \cite[Theorems 2.12, 2.14, and 5.9]{HSd}, ${\bf D}$ is a two-sided ideal of $\beta S$.
(Alternatively, it is a routine elementary argument to
establish directly that ${\bf D}$ is a two-sided ideal of $\beta S$.)
Since ${\bf D}$ is a two-sided ideal of $\beta S$, $K(\beta S)\subseteq{\bf D}$.
Since $A$ is piecewise syndetic, $\overline A\cap K(\beta S)\neq\emp$ and 
therefore $\overline A\cap {\bf D}\neq\emp$.  Therefore $d(A)>0$.

Finally, it is routine to establish that conclusion (5) holds.  To emphasize how
routine it is, we present the full proof now.
By conclusions (3) and (4), $d\big(\bigcup_{i=1}^n(x_i+B)\big)\leq\sum_{i=1}^nd(x_i+B)=nd(B)$.
Let $a=d(B)$. We complete the proof by showing that if $\alpha<a$,
then $d\big(\bigcup_{i=1}^n(x_i+B)\big)\geq n\alpha$.  

Let $\alpha<a$ be given and let $\gamma=(a-\alpha)/2$.  To see that
$d\big(\bigcup_{i=1}^n(x_i+B)\big)\geq n\alpha$, let $H\in\pf(S)$ and $\epsilon>0$ be
given.  Let $H'=H\cup\{x_1,x_2,\ldots,x_n\}$ and let $\epsilon'=\min\{\epsilon,\gamma\}$.
Since $d(B)>\alpha+\gamma$, pick $K\in\pf(S)$ such that $|B\cap K|\geq(\alpha+\gamma)|K|$
and $(\forall s\in H')(|K\symdif (s+K)|<\epsilon'|K|)$.  Then 
$(\forall s\in H)(|K\symdif (s+K)|<\epsilon|K|)$ so it suffices to show that
$\big|\big(\bigcup_{i=1}^n(x_i+B)\big)\cap K\big|\geq n\alpha |K|$.  By the disjointness assumption,
$\big|\big(\bigcup_{i=1}^n(x_i+B)\big)\cap K\big|=\sum_{i=1}^n|(x_i+B)\cap K|$ so it suffices to
let $i\in\nhat{n}$ and show that $|(x_i+B)\cap K|\geq \alpha|K|$.

So let $i\in\nhat{n}$ be given.  Since $x_i\in H'$, we have that
$|(x_i+K)\setminus K|<\epsilon'|K|$.  Therefore
  \begin{align*}
   |(x_i+B)\cap K| & \geq |(x_i+B)\cap (x_i+K)|-|(x_i+K)\setminus K|\\
    & > |B \cap K| - \epsilon'|K|\\
    & \geq (\alpha + \gamma)|K| - \gamma|K|\\
    & = \alpha|K|. \qedhere
  \end{align*}
\end{proof}

For the rest of this section we will assume we have a density function $d$ as guaranteed by
Lemma~\ref{dense}.

The next few lemmas are based on very similar results from \cite{BHL}.
Recall that, given a commutative semigroup $(S,+)$, $A\subseteq S$ and $k\in\ben$,
$kA=A+A+\cdots+A$ ($k$ times) and  $k\cdot A=
\{k\cdot s:s\in A\}$.

\begin{lemma}\label{syndetic}  Let $(G,+)$ be a commutative group and let
$A$ be a piecewise syndetic subset of $G$.  Then
$A-A$ is syndetic in $G$.\end{lemma}

\begin{proof} Pick $H\in\pf(G)$ such that 
\[
\textstyle\big(\forall F\in\pf(G)\big)
(\exists y\in G)\big(F+y\subseteq\bigcup_{t\in H}(-t+A)\big).
\]
Let $L=H-H$.  We claim that $G\subseteq\bigcup_{s\in L}\big(-s+(A-A)\big)$.
To see this, let $x\in G$. Pick $y\in G$ such that $\{0,x\}+y\subseteq\bigcup_{t\in H}(-t+A)$.
Pick $t_1$ and $t_2$ in $H$ such that
$t_1+x+y\in A$ and $t_2+0+y\in A$.  Then $x+(t_1-t_2)\in A-A$.\end{proof}

\begin{lemma}\label{subgroup}  Let $(G,+)$ be a commutative group and let
$S\subseteq G$ such that $0\in S$, $S=-S$, and $d(S)>0$.
Then there is a subgroup $E$ of $G$ such that if $l\geq 2/d(S)$, then
$lS=E$.\end{lemma}

\begin{proof}  We will show that there is some $j \leq 1/d(S)$ such that
$(2j+1)S=(2j)S$, and so $(k+1)S=kS$ for $k\geq 2 / d(S)$.  
Once we have shown this, let $k = \lceil 2/d(S) \rceil$ and let $E = kS$.
We have that $E + E = kS + kS = (2k)S = E$.  Since $S$ is symmetric, we also have that $E=-E$, 
so $E$ is closed under addition and the taking of inverses, 
hence is a subgroup of $G$ as required.

Suppose instead that for each $j \leq 1/d(S)$,
$(2j)S\subsetneq (2j+1)S$. We claim that, for each such $j$,
$(2j+1)S$ contains $j+1$ disjoint translates of $S$.  
This is a contradiction for $j=\lfloor1/d(S)\rfloor$.

The claim is true for $j=0$.  For $j>0$, choose $x \in (2j+1)S \setminus (2j)S$.
Then $x=s_1 + \cdots + s_{2j+1}$ with $s_i\in S$ for each $i$.  We have
 \begin{align*}
  S + x - s_1 = S + s_2 + \cdots + s_{2j+1} & \subseteq (2j+1)S \\
 \text{and }(2j-1)S - s_1 \subseteq (2j)S & \subseteq (2j+1)S.
 \end{align*}
Since $(2j-1)S-s_1$ contains $j$ disjoint translates of $S$, 
it suffices to show that $S + x - s_1$ and $(2j-1)S - s_1$ are disjoint.  
But if they intersect then $t_0 + x - s_1 = t_1 + \cdots + t_{2j-1} - s_1$ for some $t_i \in S$, 
from which it follows that $x = t_1 + \cdots + t_{2j-1} - t_0 \in (2j)S$, contradicting the choice of $x$.
\end{proof}

\begin{lemma}\label{sminusks} Let $(G,+)$ be a commutative group and let
$S\subseteq G$ such that $0\in S$ and $d(S)>0$.
Then there exists $Y\subseteq G$ such that for $k\geq 2/d(S)$, we
have $S-kS=Y$.\end{lemma}

\begin{proof} We suppose to the contrary that $S-(2j)S \subsetneq S-(2j+1)S$
for all $j \leq 1/d(S)$ and show that $S-(2j+1)S$ contains
$j+1$ disjoint translates of $S$ for each such  $j$, which is impossible for $j=\lfloor 1/d(S)\rfloor$.

The claim is true for $j=0$.  For $j>0$, choose $x \in \big(S-(2j+1)S\big) \setminus \big(S-(2j)S\big)$.  
Then $x=s_0 - s_1 - \cdots - s_{2j+1}$ with $s_i \in S$ for each $i$.  We have
 \begin{align*}
 S + x - s_0 & \subseteq S - (2j+1)S, \text{ and}\\
 S - (2j-1)S - s_0 & \subseteq  S - (2j)S \subseteq S - (2j+1)S\,,
 \end{align*}
hence it suffices to show that $S + x - s_0$ and $S - (2j-1)S - s_0$ are disjoint.
But if they intersect then $t_0 + x - s_0 = t_1 - t_2 - \cdots - t_{2j} - s_0$ for
some $t_i\in S$, whence $x = t_1 - t_2 - \cdots - t_{2j} - t_0 \in S - (2j)S$, 
contradicting the choice of $x$.\end{proof}

\begin{lemma}\label{aminusa}   Let $(G,+)$ be a commutative group and let
$A\subseteq G$ such that $d(A)>0$.  Assume that $E$ is a subgroup of $G$ such
that if $l\geq 2/d(A-A)$, then $l(A-A)=E$.  If $k\geq 2/d(A)$, then
$(A-kA)=(A-kA)+E$.\end{lemma}

\begin{proof} Let $k=\lceil 2/d(A)\rceil$ and let $X=A-kA$.   For any $a\in A$,
 $0\in A-a$ and $d(A-a)=d(A)$, so by Lemma \ref{sminusks}, if
$Y=(A-a)-k(A-a)$ then also $Y=(A-a)-(k+1)(A-a)$ and so
$X=X-(A-a)$.  Letting $a$ range over the whole of $A$ gives $X = X + (A-A)$. Then for all $l\in\ben$,
$X=X+l(A-A)$, hence $X=X+E$ as required.\end{proof}

The following is our promised generalisation of Lemma \ref{AminuskA}.

\begin{lemma}\label{aminuska}  Let $(G,+)$ be a commutative group and assume that
$n\cdot G$ is a central* set for each $n\in\ben$.  Let $A$ be a 
central subset of $G$.  Then $d(A)>0$ and there exists
$m\in\ben$ such that if $k\geq 2/d(A)$, then $m\cdot G\subseteq A-kA$.
\end{lemma}

\begin{proof} Since $A$ is piecewise syndetic, by Lemma \ref{dense} $d(A)>0$.
Let $S=A-A$.  Pick by Lemma \ref{subgroup} a subgroup $E$ of $G$ such that if $l\geq 2/d(S)$, then
$lS=E$.  By Lemma \ref{aminusa} we have that if $k\geq 2/d(A)$, then
$(A-kA)=(A-kA)+E$.

Now $A-A$ is syndetic by Lemma \ref{syndetic}.
Let $l=\lceil 2/d(A-A)\rceil$.  Then $A-A\subseteq l(A-A)=E$, so
$E$ is syndetic.  Pick $n\in\ben$ and $x_1,x_2,\ldots,x_n\in G$ such that
$G=\bigcup_{i=1}^n(-x_i+E)$.  For $i\in\nhat{n}$ pick $a_i<b_i$ in $\ben$ and $j\in\nhat{n}$
such that $a_i\cdot x_i\in (-x_j+E)$ and $b_i\cdot x_i\in (-x_j+E)$.
Then $(b_i-a_i)\cdot x_i\in E-E=E$.  Let $m=\prod_{i=1}^n(b_i-a_i)$.  Then
for each $i\in\nhat{n}$, $m\cdot x_i\in E$.  

We claim that $m\cdot G\subseteq E$.  So let $x\in G$ and pick $i\in\nhat{n}$ such
that $x\in (-x_i+E)$. Then $x_i+x\in E$ so $m\cdot x_i+m\cdot x\in E$.
Since $m\cdot x_i\in E$, we have that $m\cdot x\in E$.

Finally, let $k\geq 2/d(A)$.  Then
$(A-kA)=(A-kA)+E$.  Since $m\cdot G$ is central*, $A\cap m\cdot G\neq\emp$.
If $x\in A\cap m\cdot G$, then $x-k\cdot x\in (A-kA)\cap E$ so
$E=(x-k\cdot x)+E\subseteq A-kA$. Thus $m\cdot G\subseteq A-kA$.\end{proof}

\noindent {\it Proof of Lemma\/} \ref{AminuskA}. Let $C$ be a central subset
of $\ben$. By \cite[Exercise 4.3.8]{HS} $K(\beta\bez)=K(\beta\ben)\cup -K(\beta\ben)$.
Therefore $C$ is a central subset of $\bez$.
Given $n\in\ben$, by \cite[Lemma 6.6]{HS}, $n\cdot\ben$ is a member of every idempotent in $\beta\ben$, so 
in particular is a member of every idempotent in $K(\beta\ben)$.  Again using 
the fact that $K(\beta\bez)=K(\beta\ben)\cup -K(\beta\ben)$, we then have that
$n\cdot\bez$ is central* in $\bez$.  Therefore Lemma \ref{aminuska} applies.\qed

We are now ready for our strengthening of Theorem \ref{iskprn}.

\begin{corollary}\label{strongcentral} Let $A$ be the matrix of
Theorem {\rm \ref{iskprn}}. For any central subset $C$ of $\ben$,
there exist $\vec x$ and $\vec y$ in $C^\omega$ such that
$\begin{pmatrix}A&-I\end{pmatrix}\begin{pmatrix} \vec x\\ 
\vec y\end{pmatrix}=\vec 0$ and all entries of $\begin{pmatrix}\vec x \\ \vec y \end{pmatrix}$ are distinct.
\end{corollary}

\begin{proof} Let $C$ be a central set and pick an idempotent $p\in K(\beta\ben)$
such that $C\in p$.  Let $C^\star=\{x\in C:-x+C\in p\}$.  By \cite[Lemma 4.14]{HS},
if $x\in C^\star$, then $-x+C^\star\in p$.  By Lemma \ref{aminuska},
pick $m\in\ben$ such that for all $k\geq 2/d(C^\star)$, $m\cdot \bez\subseteq C^\star - kC^\star$.
Pick $M\in\ben$ such that $2^{M+1}\geq m-2+\lceil 2/d(C^\star)\rceil$.  

Let $P$ be the $(M+1)\times 2^{M+1}$ matrix consisting of rows 
$0,1,\ldots,M$ and columns $0,1,\ldots,2^{M+1}-1$ of $A$ and
let $I_{M+1}$ be the $(M+1)\times(M+1)$ identity matrix. If the 
columns of $\begin{pmatrix}I_{M+1}\\ P\end{pmatrix}$
are reversed, then the first nonzero entry in each column is $1$
so by \cite[Corollary 15.6]{HS}, $\begin{pmatrix}I_{M+1}\\ P\end{pmatrix}$
is image partition regular over $\ben$ and thus by 
\cite[Theorem 15.24(m)]{HS} we may choose $\vec x=\langle x_i\rangle_{i=0}^{2^{M+1}-1}$
and $\vec y=\langle y_i\rangle_{i=0}^M$ in $C^\star$ such that all entries of $\begin{pmatrix} \vec x \\ \vec y \end{pmatrix}$ are distinct.

Now let $r\geq M$ and assume that we have chosen $\langle x_i\rangle_{i=0}^{2^{r+1}-1}$
and $\langle y_i\rangle_{i=0}^r$ in $C^\star$ such that for each $n\in\ohat{r}$,
$y_n=2x_n+\sum_{j=2^n}^{2^{n+1}-1}x_n$ and all of the $x_i$ and $y_j$ are distinct.
For $2^{r+1} \leq t \leq 2^{r+2}-1\}$, let $z_t=x_{r+1}$.
Let $k=2^{r+1}-m+2$. Then $k\geq\lceil 2/d(C^\star)\rceil$ and
$2x_{r+1}+\sum_{t=2^{r+1}}^{2^{r+1}+m-3}z_t=m\cdot x_{r+1}\in m\cdot\bez$
so by Lemma \ref{aminuska} we may pick  $v\in C^\star$
and pick $z_t\in C^\star$ for 
$2^{r+1}+m-2 \leq t \leq 2^{r+2}-1\}$  such that
$v=2x_{r+1}+\sum_{t=2^{r+1}}^{2^{r+2}-1}z_t$.  Let
\[
  D=(-v+C^\star)\cap\bigcap_{t=2^{r+1}}^{2^{r+2}-1}(-z_t+C^\star).
\]
Then $D\in p$ which is an idempotent so by \cite[Theorem 5.8]{HS} pick a
sequence $\langle w_n\rangle_{n=1}^\infty$ in $\ben$ such that
$FS(\langle w_n\rangle_{n=1}^\infty)=\{\sum_{t\in F}x_t:F\in\pf(\ben)\}\subseteq D$.
We may assume that the sequence $\langle w_n\rangle_{n=1}^\infty$ is
increasing. Choose $n(2^{r+1})\in\ben$ such that
$w_{n(2^{r+1})}>\max\big(\big\{x_i:i\in\ohat{2^{r+1}-1}\big\}\cup\big\{y_i:i\in\ohat{r}\big\}\big)$.
Given $j\in\{2^{r+1},2^{r+1}+1,\ldots,2^{r+1}-2\}$, having
chosen $n(j)$, pick $n(j+1)>n(j)$ such that $w_{n(j+1)}>z_j+w_{n(j)}$.

For $2^{r+1} \leq j \leq 2^{r+1}-1$, let
$x_j=z_j+w_{n(j)}$ and let $y_{r+1}=2x_{r+1}+\sum_{j=2^{r+1}}^{2^{r+2}-1}x_j$.
Then each $x_j\in C^\star$ and 
$y_{r+1}=v+\sum_{j=2^{r+1}}^{2^{r+2}-1}w_{n(j)}\in C^\star$.
Finally, $\max\big(\big\{x_i:i\in\ohat{2^{r+1}-1}\big\}\cup\big\{y_i:i\in\ohat{r}\big\}\big)
<x_{2^{r+1}}<x_{2^{r+1}+1}<\cdots<x_{2^{r+2}-1}<y_{r+1}$.\end{proof}

\begin{remark}We remark that Corollary \ref{strongcentral} holds more generally. Let $G$ be a commutative group with the property
that, for every $n\in\ben$, $n\cdot G$ is a central subset of $G$. Then Corollary~\ref{strongcentral} holds with $\ben$ replaced by $G$.
The proof is essentially the same.
\end{remark}

\bibliographystyle{amsplain}

\end{document}